\documentclass[11p,reqno]{amsart}
\usepackage{amsmath}
\usepackage{amsfonts}
\usepackage{amssymb}
\usepackage{graphicx}
\usepackage{color}
\newtheorem{theorem}{Theorem}[section]
\newtheorem*{theorem*}{Theorem}

\newtheorem{corollary}[theorem]{Corollary}
\newtheorem{proposition}{Proposition}[section]

\newtheorem{conjecture}[theorem]{Conjecture}

\def\Ric{\text{Ric}}

\def\Ric{\operatorname{Ric}}
\def\Rm{\operatorname{Rm}}

\newcommand{\eps}{{\varepsilon}}

\numberwithin{equation}{section}

\begin{document}
	\title[cross quadratic bisectional curvature]{K\"ahler manifolds and cross quadratic bisectional curvature}

\author{Lei Ni}\thanks{The research is partially supported by  ``Capacity Building for Sci-Tech Innovation-Fundamental Research Funds".   }
\address{Lei Ni. Department of Mathematics, University of California, San Diego, La Jolla, CA 92093, USA}
\email{lni@math.ucsd.edu}

\author{Fangyang Zheng} \thanks{The research of FZ is partially supported by a Simons Collaboration Grant 355557.}
\address{Fangyang Zheng. Department of Mathematics,
The Ohio State University, Columbus, OH 43210,
USA}
\email{{zheng.31@osu.edu}}


\begin{abstract} In this article we continue the study of the two curvature notions for K\"ahler manifolds introduced by the first named author earlier: the so-called cross quadratic bisectional curvature (CQB) and its dual ($^d$CQB). We first show that compact K\"ahler manifolds with CQB$_1>0$ or $\mbox{}^d$CQB$_1>0$ are Fano, while nonnegative CQB$_1$ or $\mbox{}^d$CQB$_1$ leads to a Fano manifold as well, provided that the universal cover does not contain a flat de Rham factor. For the latter statement we employ the K\"ahler-Ricci flow to deform the metric. We  conjecture that all K\"ahler C-spaces will have nonnegative CQB and positive $^d$CQB.  By giving irreducible such examples with arbitrarily large second Betti numbers we show that the positivity of these two curvature put no restriction on the Betti number. A strengthened  conjecture is that any K\"ahler C-space will actually have positive CQB unless it is a ${\mathbb P}^1$ bundle. Finally we give an example of non-symmetric, irreducible K\"ahler C-space with $b_2>1$ and positive CQB, as well as compact non-locally symmetric K\"ahler manifolds with CQB$<0$ and $^d$CQB$<0$.
\end{abstract}

\maketitle

\section{Introduction}

In a recent work \cite{N} by the first named author, the concept of {\em cross quadratic bisectional curvature} (denoted as CQB from now on) and its dual notion (denoted by $\mbox{}^d$CQB) for K\"ahler manifolds were introduced (they shall be defined shortly below). Both concepts are closely related to the notion of quadratic bisectional curvature (abbreviated as QB, see \cite{WuYauZheng}, \cite{CT}, \cite{Zhang}, \cite{LiWuZheng}, \cite{ChauTam}, \cite{NiTam}, and \cite{N} for the definition and results related it). One of the  reasons for the consideration of these different notions of curvature is to find suitable differential geometric characterizations for the K\"ahler C-spaces, as motivated by the generalized Hartshorne conjecture.

In \cite{N},  inspired partially by the connection between the positive orthogonal Ricci (denoted by $Ric^\perp>0$, and  studied in \cite{NZ, NWZ}) and QB $>0$,  and partially by the work of Calabi-Vesentini \cite{CV}, among other things the first named author proved that CQB $>0$ implies $Ric^\perp>0$, which leads to the vanishing of holomorphic forms and simply-connectedness of the compact K\"ahler manifolds. The positivity of $\mbox{}^d$CQB, on the other hand, leads to the vanishing of the first cohomology group of the holomorphic tangent bundle, thus the manifold must be infinitesimally rigid, i.e., without nontrivial small deformations. It is also proved in \cite{N} that, any classical K\"ahler C-space $M^n$ with $b_2=1$ and $n\geq 2$ will have positive CQB and positive $\mbox{}^d$CQB. This makes the two conditions (namely CQB$>0$ and $\mbox{}^d$CQB$>0$) better candidates  than QB in terms of describing K\"ahler C-spaces, as only about eighty percent of the above spaces have positive or nonnegative QB by the excellent work of Chau and Tam \cite{ChauTam}.

Inspired by the perspective of a curvature characterization of the K\"ahler C-spaces, in this paper we continue  the project of understanding (with the aim of classifying) compact K\"ahler manifolds with positive or nonnegative CQB (or $\mbox{}^d$CQB). Recall that by \cite{N}, on a K\"ahler manifold $(M^n,g)$, if we denote by $T'M$ and $T''M$ the holomorphic and anti-holomorphic tangent bundle of $M$, then CQB is a Hermitian quadratic form on linear maps $A: T''M \rightarrow T'M $:
\begin{equation}\label{eq:cqb1}
 \mbox{CQB}_R(A) = \sum_{\alpha , \beta =1}^n  R(A(\overline{E}_{\alpha}), \overline{A(\overline{E}_{\alpha} )}, E_{\beta} , \overline{E}_{\beta} ) - R(E_{\alpha} , \overline{E}_{\beta}, A(\overline{E}_{\alpha}), \overline{A(\overline{E}_{\beta} )} ) \end{equation}
where $R$ is the curvature tensor of $M$ and  $\{ E_{\alpha } \}$ is a unitary frame of $T'M$. The expression is independent of the choice of the unitary frame. When the meaning is clear we simply write CQB. The manifold $(M^n,g)$ is said to have positive (nonnegative) CQB, if at any point  $x\in M$,  and for any non-trivial linear map $A: T''_xM \rightarrow T'_xM$, the value $\mbox{CQB}(A)$ is positive (nonnegative). We say that $CQB_k>0$ if (\ref{eq:cqb1}) holds for all $A$ with rank no greater than $k$.

Similarly, the dual notion ($^d$CQB) introduced in \cite{N} is a Hermitian quadratic form on linear maps $A: T'M \rightarrow T''M $:
\begin{equation}\label{eq:dcqb1}
 ^d\mbox{CQB}_R(A) = \sum_{\alpha , \beta =1}^n  R(\overline{A(E_{\alpha})}, A(E_{\alpha} ), E_{\beta} , \overline{E}_{\beta} ) + R(E_{\alpha} , \overline{E}_{\beta}, \overline{  A(E_{\alpha} )} , A(E_{\beta} ) ) \end{equation}
where $R$ again is the curvature tensor of $M$ and  $\{ E_{\alpha } \}$ is a unitary frame of $T'M$.  The manifold $(M^n,g)$ is said to have positive (or nonnegative) $^d$CQB, if $^d\mbox{CQB}(A)>0$ (or $\geq 0$) at any point in $x\in M$, and for any non-trivial linear map $A: T'_xM \rightarrow T''_xM$. Related to this there is a tensor analogous to the Ricci:  $Ric^{+}(X, \overline{X})=Ric(X, \overline{X})+H(X)/|X|^2$, where $H$ is the holomorphic sectional curvature. We say that $^d$CQB$_k>0$ if (\ref{eq:dcqb1}) holds for all $A$ with rank no greater than $k$.

It is proved in \cite{N} that compact K\"ahler manifold $M^n$ with $Ric^{+}>0$ is  projective and  simply connected. Also, if $\mbox{}^d$CQB$>0$, then $H^1(M, T'M) =\{0\}$, so $M$ is locally deformation rigid. Moreover $\mbox{}^d$CQB$_1>0$ implies $Ric^{+}>0$.

 Serving as a further step of the study our first result of this article is  that the positivity of either CQB$\mbox{}_1$ or $\mbox{}^d$CQB$\mbox{}_1$ implies the positivity of the Ricci curvature, so a compact manifold with either  CQB$\mbox{}_1>0$  or $\mbox{}^d$CQB$\mbox{}_1>0$ is  Fano, answering positively a question asked in \cite{N}.

\begin{theorem}
Let $(M,g)$ be a K\"ahler manifold with either CQB$\mbox{}_1 >0$ or $\mbox{}^d\!$CQB$\mbox{}_1 >0$. Then its Ricci curvature is positive. So compact K\"ahler manifolds with positive CQB$\mbox{}_1$ or $\mbox{}^d\!$CQB$\mbox{}_1$ are Fano.
\end{theorem}

As a corollary, the above theorem implies that a product K\"ahler manifold will have positive (or nonnegative) CQB$\mbox{}$ or $\mbox{}^d$CQB$\mbox{}$  if and only if each of its factors is so:

\begin{corollary}\label{coro:1}
Let $M=M_1\times M_2$ be a product K\"ahler manifold. Then $M$ has CQB$\mbox{} >0$ (or $\ge 0$) if and only if both $M_1$ and $M_2$ are so. Also, for any positive integer $k$, $M$ has CQB$_k\mbox{} >0$ (or $\ge 0$) if and only if both $M_1$ and $M_2$ are so. The same statements hold for $\mbox{}^d\!$CQB or $\mbox{}^d\!$CQB$\mbox{}_k$ as well.
\end{corollary}

By deforming the metric via the K\"ahler-Ricci flow we further show that  if $M$ has CQB$\mbox{}_1\ge 0$  (or $\mbox{}^d$CQB$\mbox{}_1\ge 0$) and its universal cover does not contain a flat de Rham factor then $M$ is Fano as well.  Note that the finiteness of the fundamental group of $M$ implies the nonexistence of the flat de Rham factor. Namely in particular, if $M$ has CQB$\mbox{}_1\ge 0$  (or $\mbox{}^d$CQB$\mbox{}_1\ge 0$) and $\pi_1(M)$ is finite, then $M$ is a  Fano manifold:

\begin{theorem}\label{thm:fano}
Let $(M,g)$ be a compact K\"ahler manifold with CQB$\mbox{}_1\ge 0$  (or $\mbox{}^d$CQB$\mbox{}_1\ge 0$) and its universal cover does not contain a flat de Rham factor. Then $M$ is Fano. In fact, the K\"ahler-Ricci flow will evolve the metric $g$ to ones with positive Ricci curvature.
\end{theorem}

To prove this we adopt a nice  technique of B\"ohm-Wilking \cite{BW} of deforming the metric via the K\"ahler-Ricci flow into $g(t)$ with positive Ricci curvature to our curvature conditions. In \cite{BW}, the authors deformed a Riemannian metric with nonnegative sectional curvature (also assuming finiteness of the fundamental group) into one with positive Ricci via the Ricci flow. Since  CQB$\mbox{}_1\ge0$  (or $\mbox{}^d$CQB$\mbox{}_1\ge 0$) is different from the sectional curvature being nonnegative,  a different collection of invariant time-dependent convex sets is constructed to serve the purpose. We also need  somewhat different estimates to show that $\Ric(g(t))>0$ for $t>0$, where $g(t)$ is a short time solution of the K\"ahler-Ricci flow. In  fact our curvature conditions here are much weaker than the bisectional curvature being nonnegative (which is weaker than the sectional curvature), in view of the result of Mok \cite{Mok} asserting  that the nonnegativity of bisectional curvature implies that the irreducible K\"ahler manifold is locally Hermitian symmetric, and that the first author proved in \cite{N} that all classical K\"ahler $C$-spaces with $b_2=1$ admits Einstein metrics with CQB$>0$ and $\mbox{}^d$CQB$>0$ (see also further examples with $b_2>1$ in this paper).

As suggested by Professor Richard Hamilton, the condition CQB$\geq 0$ and $^d$CQB$\geq 0$ have their analogous versions on Riemannian manifolds, and the above theorem also holds in that case. See \S 3 for more details.

By the structure theorem for compact K\"ahler manifolds with nonnegative Ricci \cite{CDP}, we have the following:

\begin{corollary}\label{coro:2}
Let $(M,g)$ be a compact K\"ahler manifold with CQB$\mbox{}_1\ge 0$  (or $\mbox{}^d$CQB$\mbox{}_1\ge 0$). Then there exists a finite cover of $M'$ of $M$, such that $M'$ is a holomorphic and metric fiber bundle over its Albanese variety, which is a flat complex torus, with the fiber being a Fano manifold.
\end{corollary}

Note that for a compact K\"ahler manifold with nonnegative QB, any harmonic $(1,1)$ form is parallel, and the positivity of QB implies that $b_2=1$. The positivity/nonnegativity of CQB or $\mbox{}^d$CQB seems not to put any restrictions on $b_2$ (see Theorem \ref{thm:ex1} below).
However, since  CQB$>0$ implies positive Ric$^{\perp}$ by \cite{N}, while ${\mathbb P}^1$ bundles do not admit any K\"ahler metric with positive $\mbox{Ric}^{\perp}$ by \cite{NWZ}, so for K\"ahler C-spaces with $b_2>1$, we could only hope for nonnegative CQB instead of positive CQB in general. We propose the following:

\begin{conjecture}\label{eq:necessary}
Any K\"ahler C-space will have nonnegative CQB and positive $^d\!$CQB.
\end{conjecture}

A slightly  weaker statement would be: any K\"ahler-Einstein C-space will have nonnegative CQB and positive  $\mbox{}^d\!$CQB. As a supporting evidence to Conjecture \ref{eq:necessary}, we prove the following:

\begin{theorem}\label{thm:ex1}
There are irreducible K\"ahler C-spaces with arbitrarily large $b_2$ which have nonnegative CQB and positive $\mbox{}^d\!$CQB.
\end{theorem}

  To prove this result as an initial study towards the conjecture, we look into the simplest kind of irreducible K\"ahler C-spaces with $b_2>1$, namely, Type $A$ flag manifolds: $M^n = SU(r+1)/ {\mathbb  T}$, where ${\mathbb T}$ is a maximal torus in $SU(r+1)$. The complex dimension is $n= \frac{1}{2}r(r+1)$ and $b_2=r$. Equip $M^n$  with the K\"ahler-Einstein metric $g$, we show that it has nonnegative CQB and positive $\mbox{}^d$CQB. This answers negatively another question asked in \cite{N} regarding $b_2$.

As shown in \cite{NWZ}, any ${\mathbb P}^1$ bundle cannot admit a K\"ahler metric with positive
 orthogonal Ricci curvature, thus cannot have positive CQB. We speculate that {\it
any K\"ahler C-space which is not a ${\mathbb P}^1$ bundle has a metric with positive CQB.}

For compact Hermitian symmetric spaces, this speculation holds true (see Corollary 2.3 in the next section). For non-symmetric K\"ahler C-spaces, result below gives at least one example of irreducible K\"ahler C-space of $b_2>1$ with positive CQB. Such a space is necessarily  not a ${\mathbb P}^1$ bundle.

Consider irreducible K\"ahler C-spaces of Type $A$ in general, namely, $SU(r+1)/K$, where $K$ is the centralizer of some sub-torus of ${\mathbb T}$. The smallest dimensional such space which is not a ${\mathbb P}^1$ bundle and not symmetric is $M^{12} = SU(6)/S(U(2)\!\times\! U(2) \! \times \! U(2))$. It has $b_2=2$. Equip it with the K\"ahler-Einstein metric, we show that it indeed has positive CQB:

\begin{theorem}\label{thm:ex2}
Let $M^{12}=$ $SU(6)/S(U(2)\times U(2)  \times  U(2))$ be the irreducible K\"ahler C-space which is non-symmetric, with $b_2=2$, and equip it with the K\"ahler-Einstein metric. Then it has positive CQB and positive $^d\!$CQB.
\end{theorem}

We should point out that understanding the curvature behavior of K\"ahler C-spaces is a nontrivial matter, despite the fact that such spaces are classical objects of study since 1950s and are fully classified from the Lie algebraic point of view. As an illustrating example, recall the following general belief:

\begin{conjecture}
Any K\"ahler C-space has positive holomorphic sectional curvature $H$.
\end{conjecture}

 This question is still widely open. For K\"ahler C-spaces with $b_2=1$, all the classical types plus a few exceptional ones are known to have $H>0$ by the work of Itoh \cite{Itoh}. In a recent PhD thesis \cite{Lohove}, Simon Lohove underwent a highly sophisticated approach and he was able to show  that all irreducible K\"ahler C-spaces of classical type with rank less than or equal to $4$  have $H>0$. Note that the rank here means that of the group, so all such spaces have $b_2 \leq 4$ in particular. Through isometric embedding, he also reduced the question largely to the case of flag manifolds with K\"ahler-Einstein metrics.

In the more challenging opposite direction, we propose the following:

\begin{conjecture}
Let $(M,g)$ be a K\"ahler-Einstein manifold with CQB $\geq 0$ and $^d\!$CQB $>0$. Then $M$ is biholomorphic to a K\"ahler C-space.
\end{conjecture}

This conjecture, if affirmed, would be the  first curvature characterization of compact homogeneous K\"ahler manifolds, which has been long missing but hoped for, in the direction of generalized Hartshorne conjecture. Of course an even bolder speculation would be to drop the K\"ahler-Einstein assumption in the above conjecture. The simply-connectedness, projectivity, and deformation rigidity result proved recently in \cite{N}, and Theorem \ref{thm:fano} above are positive evidences towards this conjecture. Theorem \ref{thm:fano} and Corollaries \ref{coro:1} and \ref{coro:2} also serve an initial step towards the classification conjecture as the main result of \cite{HSW} towards the classification of K\"ahler manifolds with nonnegative bisectional curvature. The examples in Theorems \ref{thm:ex1} and \ref{thm:ex2} indicate that the situation here is more delicate.

Note that most results mentioned above, except the construction of examples, hold for the non-positive cases by flipping the sign of the curvature. These results are summarized in the last section. In the last section we also show that the two dimensional Mostow-Siu example \cite{MS} had CQB$<0$ and $^d$CQB$<0$. This is a non-Hermitian symmetric example to which Theorem 4.1 of \cite{N} can be applied, hence locally deformation rigid (it is in fact strongly rigid in the sense of Siu as well). This naturally leads to the question of the role played by CQB and $^d$CQB  in the strong rigidity and holomorphicity of harmonic maps. We leave this to a future study.

\section{Cross quadratic bisectional curvature and its dual}

It is proved in \cite{N} that positive CQB$_1$ implies that the orthogonal Ricci curvature $\mbox{Ric}^{\perp}$ is positive, and CQB$_2>0$ implies that the Ricci curvature $\mbox{Ric}$ is $2$-positive, namely, the sum of any two of its eigenvalues is positive. We first show that the Ricci curvature is also positive under the CQB$_1>0$ assumption:

\begin{theorem}\label{thm:21}
Let $(M^n,g)$ ($n\geq 2$) be a K\"ahler manifold with positive (or nonnegative) $\mbox{CQB}_1$, then its Ricci curvature is also positive (or nonnegative). Moreover $Ric(X,\overline{X})\ge \frac{1}{n-1}Ric^\perp(X, \overline{X})$.
\end{theorem}

\begin{proof}
First we claim that, under the assumption that CQB$_1$ is positive, then for any unit vectors $X$, $Y$ in $T'M$ such that $X\perp Y$, we must have $\mbox{Ric}(X, \overline{X}) > R(X, \overline{X}, Y, \overline{Y})$. To see this, let $E$ be a unitary frame for $T'M$ with $X=E_1$ and $Y=E_2$, and let $A$ be the map such that $A(\overline{E}_2)=E_1$ and $A(\overline{E}_i)=0$ for any $i\neq 2$. Applying $(2,1)$ we get
$ \mbox{Ric}_{1\overline{1}} > R_{2\overline{2}1\overline{1}}$,
so the claim is proved. By the same token, $ \mbox{Ric}_{1\overline{1}} > R_{i\overline{i}1\overline{1}}$ for any $i>1$. Add up these inequalities for $i$ from $2$ to $n$,  we get
$ (n-1)\mbox{Ric}_{1\overline{1}} > \mbox{Ric}^{\perp}_{1\overline{1}}$,
so the Ricci curvature is positive since the orthogonal Ricci is known to be positive by \cite{N}. The nonnegative case goes similarly.
\end{proof}

\begin{corollary}
Let $M^n = M_1 \times M_2$ be a product K\"ahler manifold. Then $M$ has positive (or nonnegative) CQB$_k$ if and only if both $M_1$ and $M_2$ have positive (or nonnegative) CQB$_k$ for any $1\le k\le n$.
\end{corollary}

\begin{proof}Since CQB is independent of the choice of the unitary frames $E$ we take the unitary frame $E$ to be compatible with the product structure:
$$ E = \{ E_1, \ldots , E_r; E_{r+1}, \ldots , E_n\},$$
where  $r$ is the dimension of $M_1$ and the first $r$ elements give a frame for $M_1$. We will use the index convention that $i$, $j, \ldots $ run from $1$ and $r$, while $\alpha$, $\beta , \ldots $ run from $r+1$ and $n$. Denote by $R'$, $R''$ the curvature tensor of $M_1$, $M_2$, respectively, and write
$$ A(\overline{E}_i) = A'(\overline{E}_i) + B(\overline{E}_i), \ \ A(\overline{E}_{\alpha}) = C(\overline{E}_{\alpha}) + A''(\overline{E}_{\alpha})$$
for the decomposition into $T'M = T'M_1 \times T'M_2$, then by definition, we have
\begin{eqnarray*}
\mbox{CQB}^{M} (A) & = & \sum_{a,b,c=1}^n Ric_{a\overline{b}}A_{ca}\overline{A_{cb}} \ - \sum_{a,b,c,d=1}^n R_{a\overline{b}c\overline{d}} A_{ac} \overline{A_{bd}} \\
& = & \sum_{i,j,c}  Ric_{i\overline{j}}A_{ci}\overline{A_{cj}} + \sum_{\alpha , \beta ,c}  Ric_{\alpha\overline{\beta}}A_{c\alpha }\overline{A_{c\beta}} \ - \sum_{a,b,c,d=1}^n R_{a\overline{b}c\overline{d}} A_{ac} \overline{A_{bd}} \\
& = & \mbox{CQB}^{M_1}(A') + \mbox{CQB}^{M_2}(A'') + \sum_{i,j,\alpha} Ric_{i\overline{j}} A_{\alpha i} \overline{A_{\alpha j}} + \sum_{\alpha , \beta , i} Ric_{\alpha\overline{\beta}} A_{i\alpha } \overline{A_{i\beta }},
\end{eqnarray*}
so the conclusion follows. Note that the positivity of $CQB_k$ implies that the dimension of the manifold must be at least $2$.
\end{proof}

Since every irreducible compact Hermitian symmetric space with dimension bigger than one has positive CQB and $^d$CQB by \cite{N}, the above corollary allows us to conclude that

\begin{corollary}
Every compact Hermitian symmetric has  positive $^d\!$CQB and nonnegative CQB, and it has positive CQB if and only if it does not have any ${\mathbb P}^1$ factor.
\end{corollary}

If $(M^n,g)$ is a compact K\"ahler manifold with nonnegative Ricci curvature, then by the  work of Campana, Demailly and Peternell \cite{CDP}, the universal cover $\widetilde{M}$ of $M$ is holomorphically and isometrically the product ${\mathbb C}^k \times M_1 \times M_2$, where the first factor (if $k>0$) is the flat de Rham factor, and $M_1$ is Calabi-Yau (simply connected with trivial canonical line bundle), while $M_2$ is rationally connected. Also, there exists a finite cover $M'$ of $M$, such that the Albanese map $\pi : M' \rightarrow \mbox{Alb}(M')$ is surjective and is a holomorphic and metric fiber bundle with fiber $M_1\times M_2$. Here the bundle being metric means that any point in the base is contained in a neighborhood over which the bundle is isometric to the product of the fiber with the base neighborhood.

Now if $(M^n,g)$ is a compact K\"ahler manifold with CQB$_1\ge 0$, then since it has nonnegative Ricci, the above structure theorem applies. We claim that the Calabi-Yau factor cannot occur in this case:

\begin{theorem}\label{thm:24}
Let $(M^n,g)$ be a compact K\"ahler manifold with CQB$_1\ge 0$. Then a finite cover $M'$ of $M$ is a holomorphic and metric fiber bundle over its Albanese torus, with fiber being a rationally connected manifold. In particular, if $M$ has no flat de Rham factor, then it is rationally connected.
\end{theorem}

\begin{proof}
The goal is to rule out the Calabi-Yau factor, namely, to show that if $M_1$ is a simply-connected compact complex manifold with $c_1=0$, then it cannot admit any K\"ahler metric with CQB$_1\ge 0$. To see this, notice that we have shown that $(n-1)\mbox{Ric} \geq \mbox{Ric}^{\perp} \geq 0$. So if $\mbox{Ric} (X, \overline{X})=0$ for $X\in T'M_1$, then $\mbox{Ric}^{\perp} (X, \overline{X})=0$ and $R(X, \overline{X}, X, \overline{X})=0$. Let $\eta$ be the Ricci $(1,1)$-form of $M_1$, then by
$$ c_1\cdot [\omega ]^{n-1} = \int_{M_1} \eta \wedge \omega^{n-1} = \frac{1}{n} \int_{M_1} S \omega^n ,$$
where $\omega$ is the K\"ahler form and $S$ the scalar curvature, we see that the vanishing of the first Chern class $c_1$ plus the nonnegativity of Ricci imply that $M_1$ has to be scalar flat hence Ricci flat. So the holomorphic sectional curvature is identically zero, contradicting the fact that $M_1$ is simply connected.
\end{proof}

In fact for any pair of $X$ and $Y$ by choosing $\{E_i\}$ such that $E_1=\frac{X}{|X|}$, and letting $A$ be the map with $A(\overline{E}_1)=Y$, $A(\overline{E}_i)=0$ for $i\ge 2$ the argument above implies the following corollary.
\begin{corollary} The assumption CQB$_1\ge 0$ is equivalent to that for any $X$ and $Y$,
\begin{equation}\label{eq:bi}
|X|^2 Ric(Y, \overline{Y})-R(X, \overline{X}, Y, \overline{Y})\ge 0.
\end{equation}
If CQB$_1>0$, then the above holds as a strict inequality if $X, Y$ are nonzero.
\end{corollary}
\noindent {\em Remark: }  It is not hard to see that under the CQB $\geq 0$ assumption, any tangent vector $X\in T'M$ with $\mbox{Ric}(X,\overline{X})=0$ must be in the kernel of the curvature tensor $R$, namely, $R( X, \overline{Y}, Z, \overline{W}) =0$ for any $Y$, $Z$, $W \in T'M$.

Next, let us recall the notion of {\em dual cross quadratic bisectional curvature} ($^d$CQB) introduced in \cite{N}. It is a Hermitian quadratic form on linear maps $A: T'M \rightarrow T''M $:
\begin{equation}
 ^d\mbox{CQB}(A) = \sum_{\alpha , \beta =1}^n  R(\overline{A(E_{\alpha})}, A(E_{\alpha} ), E_{\beta} , \overline{E}_{\beta} ) + R(E_{\alpha} , \overline{E}_{\beta}, \overline{  A(E_{\alpha} )} , A(E_{\beta} ) ) \end{equation}
where $R$ again is the curvature tensor of $M$ and  $\{ E_{\alpha } \}$ is a unitary frame of $T'M$. The manifold $(M^n,g)$ is said to have positive (or nonnegative) $^d$CQB, if at any point in $M$, for any unitary frame $E$ of $T'M$ at $p$, and for any non-trivial linear map $A: T'M \rightarrow T''M$, the value $^d\mbox{CQB}_E(A)$ is positive (or nonnegative). Related to this there is a $Ric^{+}(X, \overline{X})=Ric(X, \overline{X})+H(X)/|X|^2$.

It is proved in \cite{N} that compact K\"ahler manifold $M^n$ with positive $Ric^{+}>0$ is  projective and  simply connected. If $\mbox{}^d$CQB$>0$ it also  satisfies $H^1(M, T'M) =\{0\}$, so it is locally deformation rigid. Moreover $\mbox{}^d$CQB$_1>0$ implies $Ric^{+}>0$. Strictly analogous to the nonnegative CQB case, we have the following

\begin{theorem}
A K\"ahler manifold with positive (or nonnegative) $\mbox{}^d\!$CQB$_1>0$ will have positive (or nonnegative) Ricci. A compact K\"ahler manifold with nonnegative $\mbox{}^d\!$CQB$_1\ge 0$ and without flat de Rham factor is rationally connected. Moreover $Ric(X, \overline{X})\ge \frac{1}{n+1}Ric^{+}(X, \overline{X})$. In fact $^d\!$CQB$_1\ge 0$ is equivalent to the  estimate:
\begin{equation}\label{eq:bi2}
|X|^2 Ric(Y, \overline{Y})+R(X, \overline{X}, Y, \overline{Y})\ge 0
\end{equation}
for any pair of $(1,0)$-type vectors $X, Y$. If $^d\!$CQB$_1>0$, then the above holds as a strict inequality if $X, Y$ are nonzero.
\end{theorem}

\begin{corollary}\label{coro:22} 
Let $M^n = M_1 \times M_2$ be a product K\"ahler manifold. Then for any $1\le k\le n$, $M$ has positive (or nonnegative) $\mbox{}^d\!$CQB$_k$ if and only if both $M_1$ and $M_2$ have positive (or nonnegative) $\mbox{}^d\!$CQB$_k$.
\end{corollary}

As noted in \cite{N}, when $(M^n,g)$ is K\"ahler-Einstein, the CQB or $^d$CQB conditions are given by the eigenvalue information for the curvature operator $Q$ introduced by Calabi-Vessentini \cite{CV} and Itoh \cite{Itoh}, which is the adjoint operator from $S^2(T'M)$ into itself, defined by
$$ \langle Q(X\!\cdot \!Y), \overline{Z\!\cdot \!W} \rangle = R(X, \overline{Z}, Y ,\overline{W})  $$
for any type $(1,0)$ tangent vectors $X$, $Y$, $Z$, $W$ in $T'M$, where $X\!\cdot \!Y = \frac{1}{2}(X\otimes Y + Y\otimes  X)$ and the induced metric on $S^2(T'M)$ is given by
$$ \langle X\!\cdot \!Y, \ \overline{Z\!\cdot \!W} \rangle = \frac{1}{2} \left( g(X, \overline{Z}) g(Y, \overline{W}) + g(X, \overline{Z}) g(Y, \overline{W}) \right) . $$
If we denote by $\mu $ the constant Ricci curvature of $M$, and by $\lambda_1$, $\lambda_N$ the smallest and largest eigenvalue of $Q$, respectively, then
$$ \mbox{CQB} > 0  \iff \mu > \lambda _N, \ \ \ \ \mbox{and} \ \ \ \  ^d\mbox{CQB} > 0 \iff  \lambda_1 > -\mu . $$

In  section 4, we shall examine the eigenvalue bounds for the simplest kind of K\"ahler C-spaces, namely, the Type $A$ spaces, and check the sign for CQB and $^d$CQB.

\section{Fanoness of the nonflat factor}
In this section we  study further the factor in the splitting provided by Theorem \ref{thm:24}.
If we assume that the manifold $(M, g)$ in Theorem \ref{thm:24} is simply-connected we show that $M$ is a  Fano manifold. Precisely  we have the following slightly stronger result.

\begin{theorem}\label{thm:31} Assume that $(M, g)$ be a compact K\"ahler manifold with CQB$\mbox{}_1\ge 0$ (or $\mbox{}^d\!$CQB$\mbox{}_1\ge 0$). Assume that the universal cover  $\tilde{M}$ does not have a flat de Rham factor. Then $M$ must be Fano. In fact the K\"ahler-Ricci flow evolves the metric $g$ into a K\"ahler metric $g(t)_{t\in(0, \epsilon)}$ with positive Ricci curvature for some $\epsilon$.
\end{theorem}
\begin{proof} Here we adapt an idea of B\"ohm-Wilking in \cite{BW} where the authors proved that the Ricci flow deformation of a metric with nonnegative sectional curvature of a compact manifold with finite fundamental group evolves the initial metric into one with positive Ricci curvature for some short time. The assumption on the fundamental group  is to effectively rule out the flat de Rham factor in its universal cover. A dynamic version of Hamilton's maximum principle (cf. \S 1 of \cite{BW}, Chapter 10 of \cite{Chowetc}, as well as \cite{Ni-12}) was employed. Since  CQB$\mbox{}_1\ge0$  (or $\mbox{}^d$CQB$\mbox{}_1\ge 0$) is different from the sectional curvature being nonnegative,  we need to construct a different collection of invariant time-dependent convex sets and prove the corresponding estimates to show that $\Ric(g(t))>0$. We shall focus on the case  CQB$\mbox{}_1\ge0$ since the other case is similar.

Let $g(t)$ be the solution to K\"ahler-Ricci flow with initial metric $g$ satisfying CQB$\mbox{}_1\ge 0$:
$$
\frac{\partial}{\partial t} g_{\alpha\bar{\beta}}(t)=-R_{\alpha\bar{\beta}},\quad g(0)=g
$$
where $R_{\alpha\bar{\beta}}$ denoted the Ricci curvature of $g(t)$.
  By Hamilton's maximum principle we can focus on the study of a collection of sets $\{C(t)\}$, each  being a convex subset of the space of algebraic  curvature operators satisfying the following conditions:
\begin{eqnarray}
&& 0\le  \Ric(X,\overline{X}), \ \forall \, X\in T_{x}'M;\label{eq:31}\\
&&\left| \Ric(X, \overline{Y})-R^{g(t)}_{X\overline{Y}Z\overline{Z}}\right|^2 \le (D_1+t E_1) \Ric(X, \overline{X})\cdot \Ric(Y, \overline{Y}), \ \ \forall \, X, Y, Z, |Z|=1; \ \ \ \ \label{eq:32}\\
&&\|R\|\le D_2+tE_2.\label{eq:33}
\end{eqnarray}
Here in (\ref{eq:33}) $R$ is viewed as the curvature operator and $\|\cdot\|$ is the natural norm extended to the corresponding tensors from the K\"ahler metric on $T_x'M$.

First we need to check that the sets $C(t)$  are convex. Clearly (\ref{eq:31}) and (\ref{eq:33}) are convex conditions. For (\ref{eq:32}) let $R$ and $S$ be two K\"ahler curvature operators. We shall check that if (\ref{eq:32}) holds for $R$ and $S$ then it holds for $\eta R +(1-\eta) S$ for $\eta \in [0, 1]$. Given $Z$ with $|Z|=1$, $\Ric(X, \overline{Y})-R_{X\overline{Y}Z\overline{Z}}$ is a Hermitian symmetric form on $T_x'M$, which we denote it as $A$, and denote the corresponding one for the curvature operator $S$ as $B$. We also denote $R(X,\overline{X})$ and $R(Y, \overline{Y})$ as $a_1$ and $a_2$. Similarly we have $b_1$ and $b_2$ for the corresponding Ricci of the curvature operator $S$. Then
\begin{eqnarray*}
|\eta A +(1-\eta)B|^2 &=& \eta^2 |A|^2+ \eta(1-\eta) (A\overline{B}+B\overline{A})+(1-\eta)^2 |B|^2\\
&\le&  \eta^2 |A|^2+ 2\eta(1-\eta) |A|\, |B| +(1-\eta)^2 |B|^2\\
&\le&(D_1+t E_1)\left(\eta^2 a_1 a_2 +2\eta(1-\eta)\sqrt{a_1 a_2 b_1 b_2}+(1-\eta)^2 b_1 b_2\right)\\
&\le &(D_1+t E_1)\left(\eta a_1+(1-\eta) b_1\right)\left( \eta a_2+(1-\eta) b_2\right).
\end{eqnarray*}
This completes the proof of the convexity of $C(t)$.
Recall that after applying the Uhlenbeck's trick \cite{Hamilton} the K\"ahler-Ricci flow evolves the curvature tensor $R$ by the following PDE:
\begin{equation}\label{eq:34}
\left(\frac{\partial}{\partial t}-\Delta\right) R_{\alpha\bar{\beta}\gamma\bar{\delta}}=R_{\alpha \bar{\beta}p\bar{q}}R_{\gamma\bar{\delta}q\bar{p}}+R_{\alpha \bar{\delta}p\bar{q}}R_{\gamma\bar{\beta}q\bar{p}}-R_{\alpha \bar{p}\gamma \bar{q}}R_{p\bar{\beta} q\bar{\delta}}.
\end{equation}
Here computation is with respect to a unitary frame. Tracing it gives the evolution equation of the Ricci curvature:
\begin{equation}\label{eq:35}
\left(\frac{\partial}{\partial t}-\Delta\right) R_{\alpha\bar{\beta}}=R_{\alpha\bar{\beta}p\bar{q}}R_{q\bar{p}}.
\end{equation}
We shall show that the set $C(t)$ defined by (\ref{eq:31}), (\ref{eq:32}) and (\ref{eq:33}) are invariant under the equation (\ref{eq:34}) and (\ref{eq:35}). Hamilton's maximum principle (see \S 1 of \cite{BW}) allows us to drop the diffusion term in verifying the invariance.

We first show that (\ref{eq:32}) holds at $t=0$ since by Theorem \ref{thm:21} we have that (\ref{eq:31}) holds at $t=0$, and it is easy to choose $D_2$ and $E_2$ to make (\ref{eq:33}) hold if $\epsilon$ is sufficiently small. By Theorem \ref{thm:21}, in particular (\ref{eq:bi}),  we have that for any $Z$ with $|Z|=1$, $A(X, \overline{Y})\doteqdot\Ric(X,\bar{Y})-R_{X\bar{Y} Z\bar{Z}}$ is a Hermitian symmetric tensor which is nonnegative. Diagonalize $A$ with a unitary frame $\{E_i\}$ and eigenvalues $\{\lambda_i\}$.  Then we compute  that for $X=x^i E_i$ and $Y=y^jE_j$
\begin{eqnarray*}
\left| A_{i\bar{j}}x^i\bar{y}^j\right|^2 &=&\left| \sum \lambda_i x^i\bar{y}^j\right|^2 \le \sum \lambda_i |x^i|^2 \sum \lambda_j |y^j|^2\\
&=&(\Ric(X,\bar{X})-R_{X\bar{X} Z\bar{Z}})\cdot (\Ric(Y,\bar{Y})-R_{Y\bar{Y} Z\bar{Z}})\\
&\le& \sum_{i=1}^n \left(\Ric(X,\bar{X})-R_{X\bar{X} E'_i\bar{E}'_i}\right)\sum_{j=1}^n\left(\Ric(Y,\bar{Y})-R_{Y\bar{Y} E'_j\bar{E}'_j}\right) \\
&=&(n-1)^2 \Ric(X,\overline{X}) \Ric(Y, \overline{Y}).
\end{eqnarray*}
Here $\{E'_j\}$ is another unitary frame so chosen that $E'_1=Z$. Hence if we choose $D_1=(n-1)^2$ the estimate (\ref{eq:32}) holds at $t=0$.

Now we need to verify that the PDE/ODE preserves the set $C(t)$. For that we only need to prove that the time derivative of the convex condition lies inside the tangent cone of the convex set. The trick of \cite{BW} is  to chose $E_1$ sufficiently large (compared with $D_1, D_2, E_2$) to make sure that (\ref{eq:32}) stay invariant  under the PDE (\ref{eq:34}) (or the corresponding ODE $\frac{d}{dt}\Rm =\Rm^2 +\Rm^{\#}$) for $t\in [0, \epsilon]$ if $\epsilon$ is very small.
With a  suitably chosen $D_2$, it is easy to have (\ref{eq:33}). In fact we may choose $E_2=1$ if $\epsilon$ is small. For (\ref{eq:31}), if $\Ric(X, \bar{X})$ ever becomes zero for some $X$, then within $C(t)$ by (\ref{eq:32}), we have
$$
\Ric(X, \overline{Y})-R_{X\overline{Y}Z\overline{Z}}=0,\quad \forall \ Y, Z.
$$
 This then via the polarization  implies that $R_{X \overline{Y} Z \overline{W}}=0, \, \forall \, Y, Z, W$. Thus (\ref{eq:35}) implies $\frac{\partial}{\partial t}\Ric(X, \overline{X})\ge R_{X\overline{X}p\bar{q}}R_{q\bar{p}}=0$. This shows that (\ref{eq:31}) is  preserved by (\ref{eq:35}).

As in \cite{BW}, the main issue is to show that (\ref{eq:32}) is preserved under the flow, namely (\ref{eq:34}) and (\ref{eq:35}). For this it suffices to show that as long as $R$ is in $C(t)$,
\begin{equation}\label{eq:wanted}
\frac{\partial}{\partial t} \left((D_1+t E_1) \Ric(X, \overline{X})\cdot \Ric(Y, \overline{Y})-\left| \Ric(X, \overline{Y})-R_{X\overline{Y}Z\overline{Z}}\right|^2\right)\ge 0.
\end{equation}
Direct calculation shows that the left hand side of the above inequality is
\begin{eqnarray*}
 && E_1 \Ric(X, \overline{X})\cdot \Ric(Y, \overline{Y})+(D_1+t E_1)\left(R_{X\bar{X}p\bar{q}}\Ric(Y, \overline{Y})+R_{Y\bar{Y}p\bar{q}}\Ric(X, \overline{X})\right)R_{q\bar{p}}\\
 &-&2\Re \left( \left(\frac{\partial}{\partial t} \Ric(X, \overline{Y})-\frac{\partial}{\partial t} R_{X\overline{Y}Z\overline{Z}}\right) \overline{(\Ric(X, \overline{Y})-R_{X\overline{Y}Z\overline{Z}})}\right).
\end{eqnarray*}
We shall show that for $\epsilon$ small and $t\in [0, \epsilon]$ the above is nonnegative. Namely the first term dominates the rests. By (\ref{eq:32}), by letting  $\epsilon\le \frac{1}{E_1}$ (with $E_1$ to be decided later), $(D_1+tE_1)\le  2D_1$. In the mean time  $E_1$ is chosen to be large comparing with $D_1^2D_2$.  First (\ref{eq:32}), together with $|\Ric(X, \overline{Y})|\le \sqrt{\Ric(X, \overline{X})\Ric(Y,\overline{Y})}$,   imply that
\begin{equation}\label{eq:help1}
|R_{X\overline{Y}Z\overline{Z}}|\le 4D_1\sqrt{\Ric(X, \overline{X})\Ric(Y,\overline{Y})}
\end{equation}
which then implies that
$$
\left|R_{X\bar{X}p\bar{q}}R_{q\bar{p}}\right|\le 4nD_1D_2\Ric(X,\overline{X}).
$$
This, together with   $t E_1\le 1$,  implies that
\begin{equation}\label{eq:-est23}
(D_1+t E_1)\left(R_{X\bar{X}p\bar{q}}\Ric(Y, \overline{Y})+R_{Y\bar{Y}p\bar{q}}\Ric(X, \overline{X})\right)R_{q\bar{p}}\ge -16nD_1^2 D_2 \Ric(X, \overline{X})\cdot Ric(Y, \overline{Y}).
\end{equation}
To handle the term involving $\frac{\partial}{\partial t} R_{X\overline{Y}Z\overline{Z}}$ we observe the following estimates:
\begin{eqnarray}\label{eq:help2}
|R_{X\overline{U}Z\overline{W}}|&\le& 32n D_1 \sqrt{nD_2}\sqrt{\Ric(X, \overline{X})}, \\
|R_{Y\overline{U}Z\overline{W}}|&\le& 32n D_1\sqrt{nD_2}\sqrt{\Ric(Y, \overline{Y})}, \quad \forall \ U, Z, W, \, |U|=|Z|=|W|=1. \label{eq:help3}
\end{eqnarray}
These can be derived easily out of (\ref{eq:help1}) and (\ref{eq:33}). Now note that
$$
\left| \overline{(\Ric(X, \overline{Y})-R_{X\overline{Y}Z\overline{Z}})}\right|\le\sqrt{2D_1}\sqrt{\Ric(X, \overline{X})\Ric(Y,\overline{Y})}.
$$
Hence we only need to establish that
$$
\left|\left(\frac{\partial}{\partial t} \Ric(X, \overline{Y})-\frac{\partial}{\partial t} R_{X\overline{Y}Z\overline{Z}}\right)\right|\le C(D_1, D_2, n)\sqrt{\Ric(X, \overline{X})\Ric(Y,\overline{Y})}
$$
for some positive $C$ depends on $D_1, D_2$ and $n$.
By  (\ref{eq:34}) and (\ref{eq:35}) we have that
\begin{eqnarray*}\left(\frac{\partial}{\partial t} \Ric(X, \overline{Y})-\frac{\partial}{\partial t} R_{X\overline{Y}Z\overline{Z}}\right)&=&R_{X\overline{Y}Z\overline{W}}\Ric_{Z\overline{W}}-R_{Z\overline{Z}p\bar{q}}
R_{q\bar{p}X\overline{Y}}\\
&\,& -R_{Z\overline{Y}q\bar{p}}R_{p\bar{q}X\overline{Z}}+R_{Z\bar{p}X\bar{q}}R_{p\overline{Z} q\overline{Y}}.
\end{eqnarray*}
Putting Estimates (\ref{eq:help1}), (\ref{eq:help2}) and (\ref{eq:help3}) together we  have the estimate we want. Taking $E_1\ge 100 C(D_1, D_2, n)$ we  have  proved (\ref{eq:wanted}).  Hence $\{C(t)\}$ is an invariant collection of convex subsets under the K\"ahler-Ricci flow.

If for some $t\in (0, \epsilon)$, $\Ric(g(t))$ has a nontrivial kernel, the strong maximum principle (see for example, pages 675-676 of \cite{BW})  takes effect to imply that the universal cover splits a  factor according to the distribution provided by the vectors in the kernel of the Ricci curvature. The factor is flat since by (\ref{eq:32}) the kernel of $\Ric$ would be the kernel of the curvature tensor. If there exists a sequence of such $t_i\to 0$ this implies that the universal cover contains a flat De Rham factor.  This is a contradiction. Thus we have proved that $\Ric(g(t))>0$ for any $t\in (0, \epsilon')$ for some $\epsilon'$ small.
\end{proof}

An argument similar as \cite{BW} was also employed by Liu in \cite{Liu} to the non-positive setting to conclude that the deformed metric has negative Ricci curvature if the initial metric has non-positive bisectional curvature.

The conditions CQB$\ge 0$ and $^d$CQB$\ge 0$ can have their corresponding Riemannian versions: We say that a Riemannian manifold $(M^n, g)$ has CQB$^\mathcal{R}$$\ge 0$, if for any $x\in M$ and an orthonormal frame  $\{e_i\}$, it holds that
\begin{equation}\label{eq:cqb-real}
\sum_{j=1}^n \Ric(A(e_j), A(e_j))-\sum_{i, j=1}^n R(A(e_i), e_j, A(e_j), e_i)\ge 0, \ \forall \,\mbox{ linear maps } A:T_xM \to T_xM.
\end{equation}
For $^d$CQB$^\mathcal{R}$$\ge 0$ we require that
\begin{equation}\label{eq:dcqb-real}
\sum_{j=1}^n \Ric(A(e_j), A(e_j))+\sum_{i, j=1}^n R(A(e_i), e_j, A(e_j), e_i)\ge 0, \ \forall \, \mbox{ linear maps } A:T_xM \to T_xM.
\end{equation}
 If we restrict to $A$ of rank one we have similar conditions as (\ref{eq:bi}) and (\ref{eq:bi2}). Namely, CQB$^\mathcal{R}_1\ge 0$ is equivalent to
\begin{equation}\label{eq:sec}
|X|^2\Ric(Y, Y)-R(X, Y, X, Y)\ge 0.
\end{equation}
Similarly, $^d$CQB$^\mathcal{R}_1\ge 0$ is equivalent to
\begin{equation}\label{eq:sec2}
|X|^2\Ric(Y, Y)+R(X, Y, X, Y)\ge 0.
\end{equation}
It is easy to see that (\ref{eq:sec}) and (\ref{eq:sec2}) will each imply the nonnegativity of the Ricci curvature. By adapting the proof of Theorem \ref{thm:31} we have the following result.
\begin{theorem}\label{thm:32} Assume that $(M, g)$ be a compact Riemannian manifold with CQB$^\mathcal{R}\mbox{}_1\ge 0$ (or $\mbox{}^d\!$CQB$^\mathcal{R}\mbox{}_1\ge 0$). Assume that the universal cover  $\tilde{M}$ does not have a flat de Rham factor. Then $M$ admits a metric with positive Ricci. In particular its fundamental group is finite. In fact the flow evolves the metric $g$ into a metric $g(t)_{t\in(0, \epsilon)}$ with positive Ricci curvature for some $\epsilon$.
\end{theorem}

The notions of CQB$^\mathcal{R}$ and  $^d$CQB$^\mathcal{R}$ are not geometrically motivated as CQB and $^d$CQB. We are grateful to Professor Richard Hamilton for suggesting (\ref{eq:sec}) and Theorem \ref{thm:32}  to the first named author. The nonnegative/positive conditions of these curvature respects the product structure (hence there is no difficult problem of a corresponding Hopf's conjecture for these curvatures).

\begin{proposition}\label{prop:31}
Let $M^n = M_1 \times M_2$ be a product  manifold. Then for any $1\le k\le n$, $M$ has positive (or nonnegative) $\mbox{}^d\!$CQB$\mbox{}^\mathcal{R}_k$ if and only if both $M_1$ and $M_2$ have positive (or nonnegative) $\mbox{}^d\!$CQB$\mbox{}^\mathcal{R}_k$.
\end{proposition}

It is also not hard to check  CQB$^\mathcal{R}\ge 0$ and  $^d$CQB$^\mathcal{R}\ge 0$ for the locally  symmetric spaces.  A study of these conditions perhaps should begin with the homogenous Riemannian manifolds. Given that the homogenous manifolds with positive sectional curvature is quite scarce, these conditions perhaps are more inclusive. We leave the more detailed study in this direction to a future project.

\section{K\"ahler C-spaces}

Recall that K\"ahler C-spaces are the orbit spaces of the adjoint representation of compact semisimple Lie groups. Any such space is the product of simple K\"ahler C-spaces, and all simple K\"ahler C-spaces can be obtained in the following way.

Let ${\mathfrak g}$ be a simple complex Lie algebra. They are classified as the four classical sequences $A_r={\mathfrak s}{\mathfrak l}_{\,r+1}$ ($r\geq 1$), $B_r={\mathfrak s}{\mathfrak o}_{2r+1}$ ($r\geq 2$), $C_r={\mathfrak s}{\mathfrak p}_{2r}$ ($r\geq 3$), $D_r={\mathfrak s}{\mathfrak o}_{2r}$ ($(r\geq 4$) and the exceptional ones $E_6$, $E_7$, $E_8$, $F_4$ and $G_2$.

Let ${\mathfrak h}\subset {\mathfrak g}$ be a Cartan subalgebra with corresponding root system $\Delta \subset {\mathfrak h}^{\ast}$, so we have ${\mathfrak g} = {\mathfrak h} \oplus \bigoplus_{\alpha \in \Delta} {\mathbb C}E_{\alpha }$ where $E_{\alpha}$ is a root vector of $\alpha$. Let $r=\dim_{\mathbb C} {\mathfrak h}$ and fix a fundamental root system $\{ \alpha_1, \ldots , \alpha_r\}$. This gives an ordering in $\Delta$, and let $\Delta^+$, $\Delta^-$ be the set of positive or negative roots. Each $\beta \in \Delta^+$ can be expressed as $\beta = \sum_{i=1}^r n_i(\beta )\alpha_i$. For a fixed nonempty subset  $\Phi \subseteq \{ \alpha_1, \ldots , \alpha_r\}$, denote by
$$ \Delta^+_{\Phi} = \{ \beta \in \Delta^+ \mid n_i(\beta ) > 0 \ \mbox{for \ some } \alpha_i \in \Phi \}.$$
Let $G$ be the simple complex Lie group with Lie algebra $\mathfrak g$ and $L$ the closed subgroup with Lie subalgebra $\ {\mathfrak l} = {\mathfrak h}\oplus \bigoplus_{\beta \in \Delta \setminus \Delta^+_{\Phi} }{\mathbb C}E_{\beta} $. Then $M^n= G/L$ is a simple K\"ahler C-space, and all simple K\"ahler C-spaces can be obtained that way. The complex dimension $n$ of $M$ is equal to the cardinality $|\Delta^+_{\Phi}|$, while $b_2(M) = |\Phi |$. The tangent space $T'M$ at the point $eL$ can be identified with the subspace ${\mathfrak m}^+ = \bigoplus_{\beta \in \Delta^+_{\Phi } } {\mathbb C}E_{\beta}$ of ${\mathfrak g}$. Following Itoh \cite{Itoh}, we will denote this simple K\"ahler C-space as $M^n = ({\mathfrak g}, \Phi )$.

Next let us recall the  Chevalley basis (see \cite{Helgason} or Prop. 11 of \cite{Lohove}). Let $B$ be the Killing form of $\mathfrak g$. For each $\alpha \in \Delta$, let $H_{\alpha}$ be the unique element in ${\mathfrak h}$ such that $B(H_{\alpha }, H) = \alpha (H)$ for any $H\in {\mathfrak h}$. One can always choose root vectors $E_{\alpha}$ of ${\mathfrak g}_{\alpha }$ so that $[E_{\alpha}, E_{-\alpha }] = H_{\alpha}$, $\ \overline{E}_{\alpha } = - E_{- \alpha }$, and $N_{-\alpha , -\beta } = - N_{\alpha , \beta}$, where $N_{\alpha , \beta}$ is defined by $[E_{\alpha } , E_{\beta }] = N_{\alpha , \beta } E_{\alpha + \beta}$ for any $\alpha$, $\beta \in \Delta$ with $\alpha \neq - \beta$. When $\alpha+\beta$ is not a root, then $N_{\alpha , \beta} =0$.

Denote by $z_{\alpha} = B(E_{\alpha }, E_{-\alpha })$. Then $[E_{\alpha } , E_{-\alpha }] = z_{\alpha} H_{\alpha}$, and $z_{\alpha }$ are all real and $z_{-\alpha } = z_{\alpha }$ for each $\alpha$. Now we describe the invariant K\"ahler metrics on $M$. Such a metric $g$ makes  the tangent frame $F := \{ E_{\alpha }, \alpha \in {\mathfrak m}^+\}$ an orthogonal frame, with $g( E_{\alpha } , \overline{E}_{\alpha } ) = g_{\alpha } z_{\alpha}$ where $g_{\alpha}$ satisfy the following additive condition with respect to $\Phi$:

Write $\Phi = \{ \alpha_{i_1}, \ldots , \alpha_{i_m}\}$, where $1\leq i_1<\cdots <i_m\leq r$. Assign $g_{\alpha_{i_j}}=c_j>0$ arbitrarily, and require $g_{\beta} = n_{i_1}(\beta )c_1 + \cdots + n_{i_m}(\beta ) c_m$ for any $\beta = n_1(\beta )\alpha_1 + \cdots + n_r(\beta ) \alpha_r$ in $\Delta^+_{\Phi }$. Denote this metric as $g = g_{(c_1, \ldots , c_m)}$. So the invariant K\"ahler metrics on $M$ are determined by $m=b_2$ positive constants $c_1, \ldots , c_m$. It turns out (see \S 3.2 of \cite{Lohove}) that the metric is Einstein if and only if up to scaling, $\ g_{\alpha}= \sum_{\beta \in \Delta^+_{\Phi} } B(\alpha , \beta )\ $ for any $\alpha \in \Delta^+_{\Phi}$.

Following the computation initiated in \cite{Itoh}, Lohove (\cite{Lohove}, Prop 16) completed the curvature formula for $(M^n,g)$ under the Chevalley frame $F$, which we will describe below. For $\alpha$, $\beta$, $\gamma$, $\delta \in \Delta^+_{\Phi}$, write $R(E_{\alpha }, \overline{E}_{\beta}, E_{\gamma }, \overline{E}_{\delta})$ as $R_{\alpha \overline{\beta} \gamma \overline{\delta} }$. A highly distinctive property of the curvature of $M$ is that
\begin{equation}
 R_{\alpha \overline{\beta} \gamma \overline{\delta} } = 0  \ \ \ \mbox{unless} \ \  \alpha + \gamma = \beta + \delta .
\end{equation}

To take advantage of the symmetry of curvature for K\"ahler metrics, let us consider the order relation $<$ in $\Delta$: for $\alpha \neq \beta \in \Delta$, write $\alpha < \beta$ if $n_s(\alpha ) < n_s(\beta )$ but $n_i(\alpha ) = n_i (\beta )$ for all $1\leq i<s$ (if $s>1$).

For $R_{\alpha \overline{\beta} \gamma \overline{\delta} }$ with $\alpha + \gamma = \beta + \delta$, by K\"ahler symmetries, we may assume that $\alpha$ is the smallest, and $\beta \leq \delta$. If $\alpha =\beta$, then $\gamma = \delta$, so we are left with $R_{\alpha \overline{\alpha} \gamma \overline{\gamma} }$ where $\alpha \leq \gamma $. If $\alpha \neq \beta$, then we are left with the case $\alpha < \beta \leq \delta < \gamma$. In the first case, Lohove obtained that, for any $\alpha$, $\gamma \in \Delta^+_{\Phi}$ with $\alpha \leq \gamma$:
\begin{equation}
 R_{\alpha \overline{\alpha} \gamma \overline{\gamma} } = \left\{ \begin{array}{ll} g_{\alpha } z_{\alpha} z_{\gamma} B(H_{\alpha } , H_{\gamma }) + \frac{g_{\alpha } g_{\gamma } } { g_{\alpha +\gamma } }  z_{\alpha + \gamma } N^2_{\alpha , \gamma }, \ \ \ \ \mbox{if} \ \gamma - \alpha \in \Delta^+_{\Phi}; \\    g_{\gamma } z_{\alpha} z_{\gamma} B(H_{\alpha } , H_{\gamma }) + \frac{ g_{\gamma }^2 } { g_{\alpha +\gamma } }  z_{\alpha + \gamma } N^2_{\alpha , \gamma }, \ \ \ \ \mbox{if} \ \gamma - \alpha \notin \Delta^+_{\Phi}   \end{array} \right.
\end{equation}
For the second case, he obtained that, for any $\alpha $, $\beta$, $\gamma $, $\delta \in \Delta^+_{\Phi}$ with $\alpha < \beta \leq \delta < \gamma$ and $\alpha + \gamma = \beta + \delta$,
\begin{equation}
 R_{\alpha \overline{\beta} \gamma \overline{\delta} } = \left\{ \begin{array}{ll} g_{\alpha } z_{\alpha - \beta}  N_{\alpha , -\beta } N_{\gamma , -\delta}  + \frac{g_{\alpha } g_{\beta } } { g_{\alpha +\gamma } }  z_{\alpha + \gamma } N_{\alpha , \gamma } N_{\beta , \delta }, \ \ \ \ \mbox{if} \ \gamma - \beta \in \Delta^+_{\Phi}; \\    g_{\delta } z_{\alpha - \beta}  N_{\alpha , -\beta } N_{\gamma , -\delta}  + \frac{g_{\gamma } g_{\delta } } { g_{\alpha +\gamma } }  z_{\alpha + \gamma } N_{\alpha , \gamma } N_{\beta , \delta }, \ \ \ \  \mbox{if} \ \gamma - \beta \notin \Delta^+_{\Phi}   \end{array} \right.
\end{equation}
Note that in \cite{Lohove} the curvature $R$  differs from here by a minus sign, as he is using a different sign convention.  Next let us specialize to the simplest case, namely, when
$${\mathfrak g}=A_r = {\mathfrak s}{\mathfrak l}(r+1) $$
is the space of all traceless complex $(r+1)$ square matrices. A Cartan subalgebra ${\mathfrak h}$  is given by all (traceless) diagonal matrices. The Killing form $B$ is  $B(X,Y)=\mbox{tr}(XY)$. The root system is given by
$\Delta = \{ \alpha_{ij} : 1\leq i,j\leq r+1\}$, where $\alpha_{ij}(H) = H_{ii} - H_{jj}$ for any $H\in {\mathfrak h}$, with a fundamental basis $\{ \alpha_1, \ldots , \alpha_r\}$ where $\alpha_i = \alpha_{i(i\!+\!1)}$. The positive roots are $\Delta^+ = \{ \alpha_{ij} : 1\leq i<j\leq r+1\}$, with $-\alpha_{ij} =\alpha_{ji}$.

Denote by $E_{ij}$ the $(r\!+\!1)\!\times\! (r\!+\!1)$ matrix whose only nonzero entry is $1$ at the $(i,j)$-th position, and write $H_{ij} = E_{ii}-E_{jj}$. Then $\{H_{ij}, E_{ij}\}$ forms a Chevalley basis. Since $[E_{ij}, E_{ji}] = H_{ij}$, we know that $ z_{\alpha } =1 $ for all $\alpha \in \Delta$. Thus the square norm $g(E_{\alpha }, \overline{E}_{\alpha }) = g_{\alpha }$.

To simplify our further discussions, let us introduce the following notations. For any $\alpha < \gamma $ in $\Delta^+$, we will denote by
\begin{eqnarray*}
& & \gamma \sqcup \alpha  \iff  \gamma =\alpha_{ij}, \ \alpha = \alpha_{jk}  \ \ \mbox{for  some} \ 1\leq i < j < k\leq r+1 \\
& & \gamma \sqsupset' \alpha \iff \gamma =\alpha_{ik}, \ \alpha = \alpha_{ij} \  \ \mbox{for  some} \ 1\leq i < j < k\leq r+1 \\
& & \gamma \sqsupset'' \alpha \iff \gamma =\alpha_{ik}, \ \alpha =  \alpha_{jk} \ \ \mbox{for  some} \ 1\leq i < j < k\leq r+1 \\
& & \gamma \sqsupset \alpha \iff \gamma \sqsupset' \alpha \  \ \  \mbox{or} \ \  \gamma \sqsupset'' \alpha
\end{eqnarray*}
Since $B(H_{ij}, H_{kl}) = \mbox{tr}\{ (E_{ii}-E_{jj}) (E_{kk}-E_{ll})\} = \delta_{ik}+\delta_{jl} -\delta_{il}- \delta_{jk}$,
we get $B(H_{\alpha} , H_{\alpha})=2$ for each $\alpha$, and for any $\alpha < \gamma $ in $\Delta^+$, we have
\begin{equation*}
B(H_{\alpha} , H_{\gamma}) = \left\{ \begin{array}{lll}   -1, \ \ \mbox{if} \  \gamma \sqcup \alpha \\ \ \ 1, \ \ \mbox{if} \  \gamma \sqsupset \alpha \\ \ \ 0 , \ \ \mbox{otherwise}  \end{array} \right.
\end{equation*}
Also, since $[E_{ij}, E_{kl}] = \delta_{jk}E_{il} - \delta_{il}E_{kj}$, we get that for any $\alpha < \gamma $ in $\Delta^+$,
\begin{equation*}
N_{\alpha,\gamma} = \left\{ \begin{array}{ll}   -1, \ \ \mbox{if} \  \gamma \sqcup \alpha \\  \ \ 0 , \ \ \mbox{otherwise}  \end{array} \right.
\end{equation*}
Also, for any $\alpha < \beta \in \Delta^+$,
\begin{equation*}
N_{\alpha,-\beta } = \left\{ \begin{array}{lll}   -1, \ \ \mbox{if} \  \gamma \sqsupset'' \alpha \\  \ \ 1, \ \ \mbox{if} \  \gamma \sqsupset' \alpha \\  \ \ 0 , \ \ \mbox{otherwise}  \end{array} \right.
\end{equation*}
Note that for $\delta < \gamma \in \Delta^+$, we have $N_{\gamma , -\delta } = - N_{-\gamma , \delta } = N_{\delta , - \gamma}$.
Putting all these info into the Itoh-Lahove curvature formula, we get $R_{\alpha \overline{\alpha } \alpha \overline{\alpha }} =2g_{\alpha}$, and for any $\alpha < \gamma $ in $\Delta^+_{\Phi}$,
\begin{equation*}
R_{\alpha \overline{\alpha } \gamma \overline{\gamma }} = \left\{ \begin{array}{lll}   - \frac{ g _{\alpha } g_{\gamma }} { g_{\alpha + \gamma }} , \ \ \mbox{if} \  \gamma \sqcup \alpha \\
  \ \ \ \  g_{\alpha } , \ \ \mbox{if} \  \gamma \sqsupset \alpha \\
\ \ \ \ \ 0 , \ \ \mbox{otherwise}  \end{array} \right.
\end{equation*}
Also, for $\alpha < \beta < \delta < \gamma $ in $\Delta^+_{\Phi}$ with $\alpha + \gamma = \beta +\delta $, only two cases will result in nonzero values for $R_{\alpha \overline{\beta } \gamma \overline{\delta }}$, namely,  either when $ \gamma \sqcup \alpha = \delta \sqcup \beta $ and $\gamma \sqsupset' \delta$, or when $\beta \sqsupset' \alpha $, $\delta \sqsupset'' \alpha $,  and $\gamma = \beta + \delta - \alpha$. In the first case the curvature equals to $-\frac{ g_{\alpha } g_{\delta } } {g_{\alpha + \gamma } }$, and in the second case the curvature equals to $g_{\alpha}$. Note that these two cases can be described equivalently as: there exist $1\leq i<p<q<k\leq r+1$ such that $\delta = \alpha_{ip}$, $\beta = \alpha_{pk}$,
$\gamma = \alpha_{iq}$, $\alpha = \alpha_{qk}$ for the first case, while $\delta = \alpha_{iq}$, $\beta = \alpha_{pk}$,
$\gamma = \alpha_{ik}$, $\alpha = \alpha_{pq}$ for the second case.

Now let us switch to the unitary frame $\tilde{E}_{\alpha} = \frac{ E_{\alpha } } { \sqrt{ g_{\alpha}} } $ of ${\mathfrak m}^+$. For the sake of convenience, we will still use $R_{\alpha \overline{\beta } \gamma \overline{\delta }}$ to denote the curvature component $R(\tilde{E}_\alpha ,\overline{\tilde{E}_{\beta }} ,\tilde{E}_{\gamma}  ,\overline{ \tilde{E}_{\delta }} )$. Also, to avoid clumsy notations, we will write $g_{\alpha_{ik}} $ simply as $g_{ik}$. Up to the K\"ahler symmetries, the only non-zero components of the curvature are
\begin{eqnarray}
R_{\alpha \overline{\alpha } \alpha \overline{\alpha }} & = & \frac{2}{g_{\alpha }} , \ \ \ \alpha \in \Delta^+_{\Phi} ; \\
R_{\alpha \overline{\alpha } \gamma \overline{\gamma }} & = &  \left\{ \begin{array}{ll}   - \frac{1} { g_{ik} } , \ \ \mbox{if} \ \exists \ i<j<k :  \ \gamma =\alpha_{ij}, \ \alpha = \alpha_{jk} \\
  \ \  \frac{1}{ g_{ik} } , \ \ \mbox{if} \  \ \exists \ i<j<k :  \ \gamma =\alpha_{ik}, \ \alpha = \alpha_{ij}\ \mbox{or}\ \alpha_{jk}   \end{array} \right.
\end{eqnarray}
where we assumed $\alpha < \gamma$. For $\alpha < \beta < \delta < \gamma $ in $\Delta^+_{\Phi}$, the curvature component $R_{\alpha \overline{\beta } \gamma \overline{\delta }}$ will be equal to the following non-zero values only when there are $1\leq i<p<q<k\leq r+1$  such that
 \begin{equation}
  R_{\alpha \overline{\beta } \gamma \overline{\delta }}  = \left\{ \begin{array}{ll}   - \frac{    \sqrt{  g_{ip} g_{qk} }     } {    g_{ik}   \sqrt{  g_{iq} g_{pk} }  } , \ \ \mbox{if} \   \ \delta = \alpha_{ip}, \beta = \alpha_{pk}, \gamma =\alpha_{iq}, \ \alpha = \alpha_{qk} \\
  \ \  \frac{    \sqrt{  g_{pq}  }     } {       \sqrt{  g_{ik} g_{iq} g_{pk} }  } , \ \ \mbox{if} \    \ \delta = \alpha_{iq}, \beta = \alpha_{pk}, \gamma =\alpha_{ik}, \ \alpha = \alpha_{pq}   \end{array} \right.
\end{equation}

Now check the sign for CQB or $^d$CQB. First let us consider the case when $\Phi =\{ \alpha_1, \ldots , \alpha_r\}$, namely, when $M^n = SU(r+1)/{\mathbb T}$ is the flag manifold, where ${\mathbb T}$ is a maximal torus. We have $n=\frac{1}{2} r(r+1)$, $b_2=r$, and $\Delta^+_{\Phi} =\Delta^+ $. We will choose $g$ to be the K\"ahler-Einstein metric. In this case, all $c_j=1$, and $g_{\alpha_{ik}} = k-i$. It is easy to see that the Ricci curvature is constantly equal to $\mu =2$.

For any symmetric $n\times n$ matrix $A$, the quadratic form $\langle Q(A), \overline{A} \rangle = \sum_{a,b,c,d=1}^n R_{a\overline{b}c\overline{d}} A_{ac} \overline{A_{bd}}\ $ is equals to
\begin{eqnarray}
\nonumber & & \sum_{\alpha } R_{\alpha \overline{\alpha } \alpha \overline{\alpha }} |A_{\alpha \alpha }|^2 +  \sum_{\alpha < \gamma } 4 R_{\alpha \overline{\alpha } \gamma \overline{\gamma }} |A_{\alpha \gamma }|^2 +  \sum_{\alpha < \beta < \delta < \gamma } 8\Re \{   R_{\alpha \overline{\beta } \gamma \overline{\delta }}   A_{\alpha \gamma } \overline{   A_{\beta \delta }   }     \} \\ \nonumber
& = & \sum_{\alpha } \frac{2} { g_{\alpha }}  |A_{\alpha \alpha }|^2 +  \sum_{i<j<k} \frac{4}{g_{ik}} \big( |A_{ij,ik}|^2 + |A_{jk,ik}|^2 - |A_{jk,ij}|^2  \big) + \\
& & + \sum_{i<p<q<k} 8\Re\{ -\frac{ \sqrt{g_{ip} g_{qk}  } } { g_{ik} \sqrt{  g_{iq} g_{pk} } }  A_{qk,iq} \overline{A_{pk,ip} } + \frac{ \sqrt{g_{pq}  } } {  \sqrt{  g_{ik} g_{iq} g_{pk} } }  A_{pq,ik} \overline{A_{pk,iq} } \}
\end{eqnarray}
Let us denote by $X$ and $Y$ the two terms in the last line above. We have
$$ \mbox{CQB}_{\tilde{E}}(A) = \mu ||A||^2 - \langle Q(A), \overline{A}\rangle , \ \ ^d\mbox{CQB}_{\tilde{E}}(A) = \mu ||A||^2 + \langle Q(A), \overline{A}\rangle.$$
In order to check that CQB $\geq 0$ and $^d$CQB $> 0$ for  $(SU(r+1)/{\mathbb T}, g)$, the flag manifold of type A with Einstein metric, it suffices to take care of the two crossing terms $X$ and $Y$. For $Y$, the square root part of the coefficient is less than $\frac{1}{2}$, so we have
\begin{equation*}
|Y| \leq  \sum_{i<p<q<k} 4 |A_{pq,ik} \overline{A_{pk,iq} } | \leq    \sum_{i<p<q<k}  2|A_{pq,ik} |^2 + 2| A_{pk,iq}  |^2 \end{equation*}
Note that in $2||A||^2 = 2\langle A, \overline{A} \rangle =\sum_{\alpha } |A_{\alpha \alpha }|^2 + 4\sum_{\alpha < \gamma} |A_{\alpha \gamma }|^2 $, each $|A_{pq,ik} |^2$ term or $| A_{pk,iq}  |^2$ term will appear $4$ times, so the $Y$ term will be dominated by $\mu ||A||^2$ from above or below. For the $X$ term, let us fix $i<k$ with $k-i=t+1 \geq 2$. Write $A_{ip,pk}=Z_{p}$, and write $p'=p-i$. Since the square root part of the coefficient of $X$ is less than $1$, we have
$$ |X| \leq \sum_{i<k} \sum_{1\leq p'<q'\leq t} \frac{4}{t+1} ( |Z_{p}|^2 + |Z_{q}|^2 ) = \sum_{i<p<k} \frac{4(t-1)}{t+1}|Z_p|^2  .$$
Again since for each $i<j<k$, the term $|A_{ij,jk} |^2 = |Z_j|^2$ will appear $4$ times in $\mu ||A||^2$, the $X$ term will be dominated by $\mu ||A||^2$ from above and below. Note that for the lower bound part, the term $|Z_p|^2$ will also emerge from the bisectional curvature terms, with coefficient $-\frac{4}{t+1}$. We have
$-\frac{4(t-1)}{t+1} - \frac{4}{t+1} = - \frac{4t}{t+1} > -4$, so $^d$CQB will be nonnegative, and actually positive since its vanishing would imply $A=0$. We have thus proved Theorem 1.6 stated in the introduction.

Note that if $A$ has only non-trivial entries along the diagonal line for the simple roots, then $\langle Q(A), \overline{A} \rangle =2||A||^2$, so CQB is only nonnegative and not positive.

\vspace{0.3cm}

Next let us give a non-symmetric example of irreducible K\"ahler C-space with $b_2>1$ that has positive CQB. The smallest dimensional Type $A$ space which is non-symmetric and not a ${\mathbb P}^1$ bundle would be $M^{12} = SU(6)/S(U(2)\times U(2) \times U(2))$, or equivalently, $(A_5, \Phi )= ({\mathfrak s}{\mathfrak l}_6 , \Phi )$ where $\Phi = \{ \alpha_2, \alpha_4\}$. It has $n=12$ and $b_2=2$. We have
$$ \Delta^+_{\Phi }= \{ \alpha_{kl} \mid 1\leq k<l\leq 6\} \setminus \{ \alpha_{12}, \alpha_{34} , \alpha_{56}\} .$$
Up to a scaling, the K\"ahler-Einstein metric $g$ has components  $g_{kl} = g_{\alpha_{kl}} = \sum_{\beta \in \Delta^+_{\Phi} } B(\alpha_{kl}, \beta )$, so we have
\begin{eqnarray*}
 g_{13}=g_{14}=g_{23}=g_{24} =2, \\
 g_{35}=g_{36}=g_{45}=g_{46}=2, \\
 g_{15} = g_{16} = g_{25}=g_{26} =4.
 \end{eqnarray*}
Let us denote by $\Delta_1=\{ \alpha_{15}, \alpha_{16}, \alpha_{25}, \alpha_{26} \}$ and $\Delta_2=\Delta^+_{\Phi } \setminus \Delta_1$.

So the curvature components are $R_{\alpha \overline{\alpha} \alpha \overline{\alpha}} = \frac{2}{g_{\alpha }}$, which is $  \frac{1}{2}$ for $\alpha \in \Delta_1$ and $1$ for $\alpha \in \Delta_2$. While $R_{\alpha \overline{\alpha} \gamma \overline{\gamma }}$ are given by $(4.5)$. It is easy to see that the Ricci curvature is constantly $\mu =2$ in this case. The crossing terms $R_{\alpha \overline{\beta} \gamma \overline{\delta }}$ are given by $(4.6)$. We have
$$ \mu ||A||^2 = \sum_{\alpha } 2| A_{\alpha \alpha }|^2 + \sum_{\alpha < \gamma } 4 |A_{\alpha \gamma }|^2. $$
Now consider the quadratic form $\langle Q(A), \overline{A}\rangle $ given by $(4.7)$.  Let us examine the two terms $X$ and $Y$ in the last line of $(4.7)$. For the term $Y$, note that $i<p<q<k$ could be from $(1,2,3,4)$, $(3,4,5,6)$, in which case the square root part of the coefficient is $\frac{1}{2}$,  or from $(1,2,3,5)$, $(1,2,3,6)$, $(1,2,4,5)$, $(1,2,4,6)$, $(1,2,5,6)$, $(1,3,5,6)$, $(1,4,5,6)$, $(2,3,5,6)$, or $(2,4,5,6)$. In each of these last $9$ cases the square root part of the coefficient for the $Y$ terms is $\frac{1}{4}$. So we have
\begin{equation*}
 |Y|  \leq  2 \sum_{i<p<q<k} | A_{ik, pq} \overline{A_{iq, pk}}| \leq  \sum_{i<p<q<k} ( |A_{ik,pq}|^2 + |A_{iq,pk}|^2 ).
\end{equation*}
Here in the sum we are skipping those terms with $(p,q)=(3,4)$. Note that in $\mu ||A||^2$, each of the terms $|A_{iq,pk}|^2$ appears with coefficient $4$, so $|Y|$ is strictly dominated by $\mu ||A||^2$ from above and below. Next let us consider the $X$ terms. For each of $\alpha =\alpha_{qk}$, $\beta =\alpha_{pk}$, $\gamma =\alpha_{iq}$, $\delta =\alpha_{ip}$ to be in $\Delta^+_{\Phi}$, the indices $i<p<q<k$ could only take the following four cases: $(1,3,4,5)$, $(1,3,4,6)$, $(2,3,4,5)$, $(2,3,4,6)$. In each case, the square root part of the coefficient is $1$, while $g_{ik} = 4$, so we have
$$ |X| \leq  \sum_{i=1}^2 \sum_{k=5}^6  |A_{i3,3k}|^2 + |A_{i4,4k}|^2 .$$
So each of these $|A_{i3,3k}|^2$ or $|A_{i4,4k}|^2$ term in the quadratic form will be strictly dominated by that from $\mu ||A||^2$ from both sides. The other terms are clearly strictly dominated by $\mu ||A||^2$ from both above and below. So $(M^{12},g)$ has positive CQB and positive $^d$CQB, and we have completed the proof of Theorem 1.8 stated in the introduction.

\section{Non-positive cases}

One may also consider K\"ahler manifolds with non-positive CQB or $^d$CQB.  Similar to the nonnegative cases, we have the following results:

\begin{theorem} Let $(M, g)$ be a K\"ahler manifold with CQB$\mbox{}_1\le 0$. Then for any $X, Y\in T_x'M$
\begin{equation}\label{eq:bi-2}
|X|^2 Ric(Y, \overline{Y})-R(X, \overline{X}, Y, \overline{Y})\le 0.
\end{equation}
The above holds as strict inequality (for nonzero $X$, $Y$) if CQB$\mbox{}_1<0$. In particular
$\Ric(Y,\overline{Y})\le \frac{1}{n-1}\Ric^\perp(Y,\overline{Y})\le 0$.

Similarly, if $(M,g)$ is K\"ahler with $\mbox{}^d\!$CQB$\mbox{}_1\le 0$, then for any $X, Y\in T_x'M$
\begin{equation}\label{eq:bi-2a}
|X|^2 Ric(Y, \overline{Y})+R(X, \overline{X}, Y, \overline{Y})\le 0,
\end{equation}
and the inequality is strict (for nonzero $X$, $Y$) when $\mbox{}^d\!$CQB$\mbox{}_1< 0$. In particular, it holds that
$\Ric(Y,\overline{Y})\le \frac{1}{n+1}\Ric^+(Y,\overline{Y})\le 0$.

A product K\"ahler manifold $M=M_1\times M_2$ has CQB$\mbox{}<0$ (or $\le 0$, or $\mbox{}^d\!$CQB$\mbox{}< 0$, or $\mbox{}^d\!$CQB$\mbox{}\le 0$) if and only if each factor is so. For any positive integer $k$, $M$ has CQB$\mbox{}_k$ (or $\mbox{}^d\!$CQB$\mbox{}_k$) $<0$ or $\le 0$ if and only if each  factor is so.
\end{theorem}
\begin{proof}
The proof is exactly the same as that of Theorem \ref{thm:21}.
\end{proof}

\begin{theorem}Assume that $(M, g)$ be a compact K\"ahler manifold with CQB$\mbox{}_1\le 0$ (or $\mbox{}^d\!$CQB$\mbox{}_1\le 0$). Assume that the universal cover  $\tilde{M}$ does not have a flat de Rham factor. Then $M$ must admit a metric with $\Ric<0$. In fact the K\"ahler-Ricci flow evolves the metric $g$ into a K\"ahler metric $g(t)_{t\in(0, \epsilon)}$ with negative Ricci curvature for some $\epsilon$.
\end{theorem}
\begin{proof}
We can prove the result by following the same argument and flipping the sign when needed in the proof of Theorem \ref{thm:31}.
\end{proof}

Next construct  examples of compact K\"ahler manifolds with negative (non-positive) CQB and $^d$CQB. First of all,
if $M^n$ is a compact quotient of a Hermitian symmetric space $\widetilde{M}$ of non-compact type, then by  \cite{CV}, we see that $M$ always has $\mbox{}^d$CQB $<0$ and CQB $\le 0$, and it will have CQB $<0$ when and only when $\widetilde{M}$ does not have the unit disc as an irreducible factor.

For non-locally Hermitian symmetric examples, we adapt the construction of strongly negatively curved manifolds by Mostow and Siu \cite{MS} and by the second named author \cite{Z1}, \cite{Z2}. To state the result, let us recall the notion of  {\em good coverings}.

A finite branched cover $f:M^n \rightarrow N^n$ between two compact complex manifolds is called a good cover, if for any $p\in M$, there exists locally holomorphic coordinates $(z_1, \ldots , z_n)$ centered at $p$ and $(w_1, \ldots , w_n)$ centered at $f(p)$, such that $f$ is given by $w_i=z_i^{m_i}$, $1\le i\le n$, where $m_i$ are positive integers. Note that the branching locus $B$ and ramification locus $R$ are necessarily normal crossing divisors in this case.

In \cite{MS}, Mostow and Siu computed the curvature for the Bergman metric of the Thullen domain $\{ |z_1|^{2m} + |z_2|^2<1\}$, and used it to construct examples of strongly negatively curved surfaces which is not covered by ball. In \cite{Z1}, the second named author generalized this to higher dimensions, and also at the quotient space level using the Poincar\'e distance, and showed that (see Theorem 1 of \cite{Z1}) if $N$ is a compact smooth quotient of the ball, and $B\subset N$ a smooth totally geodesic divisor (possibly disconnected), then for any good cover $f: M \rightarrow N$ branched along $B$, $M$ admits a K\"ahler metric with negative complex curvature operator. We will use this computation to claim the following:
\begin{theorem}
Let $N^n$ ($n\geq 2$) be a smooth compact quotient of the ball, equipped with the complex hyperbolic metric, and let $B\subset N$ be a smooth totally geodesic divisor (possibly disconnected). If $f: M\rightarrow N$ is a good cover branched along $B$, then $M$ admits a K\"ahler metric $g$ which has negative CQB and negative $^d\!$CQB.
\end{theorem}
\noindent {\em Remark:} Such a manifold $M$ is not homotopy equivalent to any locally Hermitian symmetric space,  and it is strongly rigid in the sense of Siu, namely, any compact K\"ahler manifold homotopy equivalent to $M$ must be (anti)biholomorphic to $M$.
\begin{proof}
The construction of the K\"ahler metrics $\omega_{\eps}$ is exactly the same as in the proof of Theorem 1 of \cite{Z1}. Notice that at the point $p$ in a tubular neighborhood $V$ of the ramification locus $R$, there exists tangent frame $e$ at $p$ such that $e_i \perp e_j$ whenever $i\neq j$, and under $e$ the only non-zero curvature components of $\omega_{\eps}$ are $R_{i\overline{i}j\overline{j}}$, with
$$ -R_{1\overline{1}1\overline{1}} =b, \ \ -R_{1\overline{1}i\overline{i}} =c, \ \ -R_{i\overline{i}i\overline{i}} =2e, \ \ -R_{i\overline{i}j\overline{j}} =e $$
for any $2\le i<j$. It was shown that $b>0$, $c>0$, $e>0$, and $nbe > (n-1)^2c^2$.

Note that if we normalize $e$, namely, replace $e_k$ by $\frac{e_k}{|e_k|}$ for each $k$, then the above inequalities on $b$, $c$, and $e$ still holds. So let us assume that $e$ is unitary at $p$. For any non-trivial $n\times n$ matrix $A$, we have
$ -\mbox{CQB}_e(A) = P - Q$,  and $-\,^d\mbox{CQB}_e(A) = P + Q$,  where
\begin{eqnarray*}
P &= & -\sum_{i,j,k,\ell} R_{i\overline{j}k\overline{k}} A_{\ell i} \overline{A_{\ell j}}
\ = \  -\sum_{i, k, \ell} R_{i\overline{i}k\overline{k}} |A_{\ell i}|^2\\
& = & (b + (n-1)c) \sum_{\ell} |A_{\ell 1}|^2 + (c+ne) \sum_{i>1, \ell}|A_{\ell i}|^2 \\
Q &= & -\sum_{i,j,k,\ell} R_{i\overline{j}k\overline{\ell}} A_{ik} \overline{A_{j\ell}}
\ = \  -\sum_{i} R_{i\overline{i}i\overline{i}} |A_{ii} |^2 - \sum_{i<k} R_{i\overline{i}k\overline{k}} |A_{ik} + A_{ki}|^2 \\
& = & b|A_{11}|^2 + 2e\sum_{i>1} |A_{ii}|^2 + c\sum_{i>1} |A_{1i}+A_{i1}|^2 + e \sum_{1<i<k} |A_{ik}+A_{ki}|^2
\end{eqnarray*}
Clearly, $P+Q>0$ for all $A\neq 0$, and if we write $t_{ij}=|A_{ij}|^2$, we have
\begin{eqnarray*}
P-Q &= & \ \ (n-1)c\, t_{11} \ + \ (c+(n-2)e)\sum_{i,k>1} t_{ik} \ + \\
& & + \sum_{i>1} \left( (b+(n-2)c)\,t_{i1} + ne \,t_{1i} - 2c \, \Re (A_{i1}\overline{A_{1i}}) \right) ,
\end{eqnarray*}
which is positive as $nbe > (n-1)c^2 > c^2$. So the metric $\omega_{\eps}$ has CQB $<0$ and $^d$CQB $<0$ in $V$, for {\em any} $\eps >0$. By choosing $\eps$ sufficiently small, one see that CQB and $^d$CQB will be  negative everywhere in $M$.
\end{proof}

By \cite{Z1} and \cite{Z2}, we see that there are many examples of such $M$ in $n=2$. An example in  $n=3$ was constructed by M. Deraux in \cite{Deraux}, and we are not aware of any higher dimensional such constructions, even though it has been widely believed that there should be plenty in all dimensions.

\section*{Acknowledgments} { We would like to thank Professors Hung-Hsi Wu for his interest to this work, Richard Hamilton for suggesting the Riemannian versions of (\ref{eq:bi}) and (\ref{eq:bi2}).
}

\end{document}